\documentclass{amsart}
\usepackage{latexsym}
\usepackage{amstext}
\usepackage{amssymb}
\usepackage{amsfonts}
\usepackage{amsthm}
\usepackage{amsopn}
\usepackage{amsbsy}
\usepackage{layout}
\usepackage{color}
\usepackage{amsmath}
\usepackage{graphicx}
\usepackage{bm}
\usepackage{yfonts}
\usepackage{mathrsfs}

\catcode `@ = 11
\renewcommand\section{\@startsection{section}{1}{\z@}%
                                   {-3.5ex \@plus -1ex \@minus -.2ex}%
                                   {2.3ex \@plus.2ex}%
                                   {\normalfont\large\bfseries}} 
\renewcommand\subsection{\@startsection{subsection}{2}{\z@}%
                                     {-3.25ex\@plus -1ex \@minus -.2ex}%
                                     {1.5ex \@plus .2ex}%
                                    {\normalfont\itshape}} 
\catcode `@ = 12

\newcommand{\dfn}[2]{{\it #1}{\index{#2}}}

\def\newtheorems{\newtheorem{theorem}{Theorem}[section]
                 \newtheorem{maintheorem}[theorem]{Main Theorem}
                 
                 \newtheorem{corollary}[theorem]{Corollary}
                 
                 \newtheorem{proposition}[theorem]{Proposition}
                 \newtheorem{lemma}[theorem]{Lemma}

                 \newtheorem{remark}[theorem]{Remark}

								 \newtheorem{problem}[theorem]{Problem}
								
}

\newtheorem*{examples*}{Examples}
\newtheorem*{example*}{Example}

\newtheorem*{remark*}{Remark}
\newtheorem*{remarks*}{Remarks}

\newtheorems

\def\rfs#1{\mbox{\rm #1}}

\newcommand{\natN}{{\mathbb N}}

\newcommand{\natQ}{{\mathbb Q}}
\newcommand{\natR}{{\mathbb R}}

\newcommand{\natZ}{{\mathbb Z}}

\newcommand{\cmP}{{\mathscr P}}

\newcommand{\cmU}{{\mathscr U}}
\newcommand{\cmV}{{\mathscr V}}

\def\powerset{{\mathscr{P}}}

\newcommand{\eqdf}{\hbox{\bf \,:=\,}}

\newcommand{\onetpl}[1]{\langle#1\kern1.4pt\rangle}
\newcommand{\pair}[2]{\langle#1 ,#2\kern1.4pt\rangle}
\newcommand{\triple}[3]{\langle#1 ,#2\kern1.4pt ,#3\kern1.4pt\rangle}
\newcommand{\quaple}[4]{\langle#1 ,#2\kern1.4pt ,#3\kern1.4pt,#4\kern1.4pt\rangle}

\newcommand{\setof}[2]{\{#1\::\:#2\}}
\newcommand{\SETOF}[2]{\bigl\{#1\colon #2\bigr\}}

\newcommand{\seqn}[2]{\langle#1\,:\,#2\rangle}

\newcommand{\abs}[1]{| #1 |}

\newcommand{\fnn}[3]{#1:#2 \to #3}

\usepackage[colorlinks=true, hyperindex=true, 
urlcolor=blue, citecolor=red, linkcolor=blue, bookmarks=true]{hyperref}

\usepackage{ulem}

\usepackage{xcolor}

\begin{document}
\title[Countable Successor Ordinals as GO-spaces]{Countable Successor Ordinals as Generalized Ordered Topological Spaces}

\author[Robert Bonnet and Arkady Leiderman]{Robert Bonnet$\,^1$ and Arkady Leiderman}
\address{R. Bonnet:
Laboratoire de Math\'ematiques,
Universit\'e de Savoie, Le Bourget-du-Lac, France}
\email{bonnet@univ-savoie.fr {\rm\ and }
Robert.Bonnet@math.cnrs.fr}
\address{A.~Leiderman:
Department of Mathematics, 
Ben-Gurion University of the Negev, Beer Sheva, Israel}
\email{arkady@math.bgu.ac.il}
\thanks{$^1\,$The first listed author gratefully acknowledges the financial support he received from the Center for Advanced Studies in Mathematics
 of the Ben-Gurion University of the Negev}

\keywords{Interval spaces, linearly ordered topological spaces, generalized ordered spaces, compact spaces}
\subjclass{03E10, 06A05, 54F05, 54F65}
\date{December 26, 2016}

\begin{abstract}
We prove the following  Main Theorem: 
Assume that any continuous image of a Hausdorff topological space $X$ is a generalized ordered space. 
Then $X$ is homeomorphic to a countable successor ordinal 
(with the order topology). 
\newline
The converse trivially holds.
\end{abstract}

\maketitle

\section{Introduction and Main Theorem}\label{introduction}
All topological spaces are assumed to be Hausdorff.
Remind that $L$ is a \dfn{Linearly Ordered Topological Space}%
{linearly ordered \\ topological space (LOTS)} (\dfn{LOTS}{LOTS})
whenever there is a linear ordering 
$\leq^L$ on the set $L$ such that a basis of the topology $\lambda^L$ on $L$ 
consists of all open convex subsets.
A \dfn{convex}{convex subset} set $C$ in a linear ordering $M$ is a subset of $M$ with the property: for every $x<z<y$ in $M$,  
if $x,y \in C$ then $z \in C$.
The above topology, denoted by \dfn{$\lambda^L$}{$\lambda^L$ (order topology)} 
is called an \dfn{order topology}{(order topology)}. 
Since the order $\leq^L$ defines the topology $\lambda^L$ on $L$ (but not vice-versa), we denote also by 
$\pair{L}{\leq^L}$ the structure including the topology $\lambda^L$. 

A topological space $\pair{X}{\tau^X}$ is called a 
\dfn{Generalized Ordered Space}{Generalized Ordered Space (GO-space)} 
(\dfn{GO-space}{GO-space})
whenever $\pair{X}{\tau^X}$ is homeomorphic to a subspace of a LOTS $\pair{L}{\lambda^L}$, 
that is $\tau^X = \lambda^L {\restriction} X \eqdf \setof{ U \cap X }{U \in \lambda^L}$ 
(see \cite{BL}). 

Evidently, every LOTS, and thus any GO-space, is a Hausdorff topological space, but not necessarily separable or Lindel\"of.
The Sorgenfrey line $Z$ is an example of a GO-space, which is not a LOTS, and such that every subspace of $Z$ is separable and Lindel\"of
(see \cite{E}).

By definition, every subspace of a GO-space is also a GO-space. 
In this article, in a less traditional manner, we say that a space $X$ is a 
\dfn{hereditarily GO-space}{hereditarily GO-space} if every continuous image of $X$ (in particular, $X$ itself) is a GO-space.

\begin{maintheorem} 
\label{main-thm}
Every hereditarily GO-space is homeomorphic to a countable successor ordinal, considered as a LOTS. 
The converse obviously holds. 
\end{maintheorem}

This result is closely related to the following line of research: characterize Hausdorff topological spaces $X$ such that
 all continuous images of $X$ have the topological  property $\cmP$. All questions listed below for concrete $\cmP$ are still open.
\newpage
\begin{problem}
\label{question-1.2}
\noindent
\begin{itemize}
\item[{\rm(1)}]
Characterize Hausdorff spaces such that all continuous images of $X$ are regular. 
\item[{\rm(2)}]
Characterize Hausdorff spaces $X$ such that all continuous images of $X$ are normal.
(This was partially solved by W. Fleissner and R. Levy in \cite{FL1, FL2}).
\item[{\rm(3)}]
Characterize Hausdorff spaces such that all continuous images of $X$ are realcompact.
(This question has been formulated in \cite{Okunev}).
\item[{\rm(4)}]
Characterize Hausdorff spaces such that all continuous images of $X$ are paracompact.
\item[{\rm(5)}]
Characterize Hausdorff spaces such that all continuous images of $X$ are monotonically normal. 
(This is related to ``Niekel Conjecture'' answered positively by M. E. Rudin \cite{Ru}).
\hfill$\square$
\end{itemize}
\end{problem}

\begin{remark}\label{R1}
{\rm
Let $\mathfrak{b}$ be the minimal cardinality of unbounded subsets of $\omega^\omega$.
Recently M. Bekkali and S. Todor\v cevi\'c proved the following relevant result:
Continuous zero-dimensional images of a compact LOTS of weight less than $\mathfrak{b}$ is itself a LOTS \cite[Theorem 4.2]{BT}.

Note also a recent paper \cite{Tkachuk}, which studies topological properties $\cmP$ that are reflectable in small continuous images.
 For instance, a GO-space $X$ is hereditarily Lindel\"of iff all
continuous images of $X$ have countable pseudocharacter \cite{Tkachuk}.
}
\end{remark}

As a special case of Main Theorem~\ref{main-thm}, we obtain the following fact. 
\begin{corollary}
\label{cor-456789123nn}
Assume that any continuous image of a Hausdorff topological space $X$ is a LOTS.
Then $X$ is homeomorphic to a countable successor ordinal. 
\end{corollary}

In order to make this paper widely readable, we have tried to give self-contained and elementary proofs,
even when our results could be deduced from more general theorems. 
By these reasons, and for the readers' convenience, we include a separate and short proof
 of Corollary~\ref{cor-456789123nn} in Section~\ref{section-elementary}. 

In our paper, a \dfn{GO-structure}{GO-structure}, formally, 
is a 4-tuple  
\hbox{$ \quaple{X}{\tau^X}{L}{\leq^L} $} 
where 
\begin{itemize}
\item[{\rm(1)}]
$X \subseteq L$,
\item[{\rm(2)}]
$\pair{L}{\leq^L}$ is a linear ordering 
with the order topology $\lambda^L$, 
and %
\item[{\rm(3)}] 
the order $\leq^X$ on $X$ is the restriction $\leq^L \restriction X$ of $\leq^L$ to $X$,
\ \  and \ \  
\\
$\tau^X$ is the topology $\lambda^L {\restriction} X 
\eqdf \setof{ U \cap X }{U \in \lambda}$
\end{itemize}
%
Any GO-space 
$X$ can be written under the above structure.
In~\cite{HNV}, \cite{L}  GO-space are denoted by 
$\triple{X}{\tau^X}{\leq^X}$ 
where $\leq^X$ is the restriction of $\leq^L$ to $X$. 
It is easy to see that the following are equivalent.
\begin{itemize}
\item[{\rm(i)}]
$\triple{X}{\tau^X}{\leq^X}$ is a GO-space and 
\item[{\rm(ii)}]
$ \lambda^X \subseteq \tau^X$ and $\tau^X$ has a base consisting of convex sets.
\end{itemize}

Many times we denote $\quaple{X}{\tau^X}{L}{\leq^L}$ 
by $\quaple{X}{\tau^X}{L}{\lambda^L}$.

The proof of Main Theorem~\ref{main-thm}, is organized as follows.
In \S\ref{section-basic}, we present the basic facts on LOTS and GO-spaces. 
In particular, Proposition~\ref{prop-fqljdlfk} shows that if $\quaple{X}{\tau^X}{L}{\leq^L}$ 
is a GO-structure, then we may assume that $L$ is a complete ordering and that $X$ is topologically dense in $\pair{L}{\lambda^L}$. 
In \S\ref{subsection preliminary} we show that any hereditarly GO-space satisfies c.c.c. property: every family of pairwise disjoint nonempty open set is countable (Lemma \ref{lemma-2.1}). 
In \S\ref{subsection counterexample} we prove that any hereditarily GO-space has no countable closed and relatively discrete subset (Corollary~\ref{cor-hfdsk}): 
we  recall that a subset $D$ of a space $Y$ is \dfn{relatively discrete}{relatively discrete (subset)} whenever there is a family $\cmU_D \eqdf \setof{ U_d }{ d \in D }$ of open subsets of $Y$ such that $D \cap U_d = \{d\}$ for every $d$.
Proposition~\ref{prop-kfjdsluzreo} shows that a hereditarily GO-space is a subspace of a scattered linear order (with the order topology).
Finally, in \S\ref{subsection proof of thm} we conclude the proof of Main Theorem.


\section{Elementary proof of Main Theorem for LOTS: Corollary~\ref{cor-456789123nn}}
\label{section-elementary}


\newcommand{\seqnn}[2]{{\langle#1\rangle}_{#2}}

We assume that the reader is familiar with the properties of LOTS and give a self-contained proof of Corollary~\ref{cor-456789123nn}. 
To prove the result, we need some preliminaries.

\noindent
{\bf Fact 1.}
\begin{it}%
Let $\pair{L}{\leq^L}$ be a LOTS. 
If $\pair{L}{\leq^L}$ is a scattered ordering then $\pair{L}{\lambda^L}$ is a scattered topological space.
\hfill$\square$
\end{it}
\smallskip

The converse does not hold: consider the lexicographic sum 
$N \eqdf \sum_{r \in \natQ} \natZ_r$  of copies $\natZ_r$ of the integers $\natZ$ over the rational chain $\natQ$: \,$\pair{N}{\leq^N}$ is not a scattered ordering but $\pair{N}{\lambda^N}$ is a discrete space (this example will be used again in  Part~(1) of Remark~\ref{R2}). 

The next fact is used implicitly in the arguments that follow.
\smallskip

\noindent
{\bf Fact 2.}
\begin{it}%
Let $\pair{L}{\leq^L}$ be a LOTS and $F$ be a closed subspace of 
$\pair{L}{\lambda^L}$. 
Then the induced topology $\lambda^L {\restriction} F$ on $F$ is the order topology on $F$ defined by the restriction ${\leq^L} 
{\restriction} F$ of $\leq^L$ on $F$.
\hfill$\square$
\end{it}
\smallskip

\noindent
{\bf Fact 3.}
\begin{it}%
Let  $\pair{L}{\leq^L}$ be a scattered linear ordering then $\pair{L}{\lambda^L}$ is $0$-dimensional (i.e. $\pair{L}{\lambda^L}$ has a base consisting of
 clopen sets). 
\end{it}
\newline
{\it Proof.} The proof uses the fact that if $\pair{L}{\leq^L}$ is a scattered linear ordering, then the Dedekind completion $\pair{L^c}{\leq^{L^c}}$ of $\pair{L}{\leq^L}$ is also a scattered linear ordering.
\hfill$\square$
\smallskip

\noindent
{\bf Fact 4.}
\begin{it}%
Let any continuous image of a Hausdorff topological space $Y$ is a LOTS. 
Then $Y$ satisfies c.c.c.\,.
In particular, $\omega_1$ and $\omega_1^*$ are not order-embeddable in $Y$. 
\end{it}
\smallskip

\noindent
{\bf Fact 5.}
\begin{it}%
Let $Y$ be a $0$-dimensional LOTS. 
If $D$ is a countable closed and discrete subset of $Y$  
then $D$ is a continuous image of $Y$.
\end{it}
\newline
{\it Proof.} 
Let $\seqnn{ U_d }{ d \in D}$ be clopen subsets of $Y$ such that  
$U_d \cap D = \{ d \}$ for $d \in D$. 
Fix $d_0 \in D$. 
Let $\approx$ be the equivalence relation on $Y$ defined by 
\smallskip
$x \approx y$ whenever there is $d \in D \setminus \{ d_0 \}$ such that $x,y \in U_d$, or
$x, y \not\in U \eqdf \bigcup \, \setof{\, U_d }{ d \in D {\setminus} \{ d_0 \} \, }$. 
Then $Y / { \approx }$ is a continuous image of $Y$, 
\ $Y / { \approx }$ is Hausdorff and $D$ is homeomorphic to  $Y / { \approx }$. 
\hfill$\square$
\smallskip

\noindent
{\bf Fact 6.}
\begin{it}%
Let $D$ be a countable and discrete space. 
Then $D$ has a continuous image which is not homeomorphic to a LOTS.
\end{it}
\newline
{\it Proof.}
Consider $\omega$ as a discrete space. 
Let $\cmU$ on $\omega$ be a non-principal ultrafilter on $\omega$ and let $*$ be a new element with $* \not\in \omega$.
We equip $\natN^* := \omega \cup \{*\}$ with the topology induced from the Stone--\v Cech compactification $\beta \natN$.
It is well-know that  $\natN^*$ is not a LOTS: 
this is so because $\natN^*$ is separable and its topology does not have a countable base \cite[\S3.6]{E}. 
Evidently, countable $\natN^*$ is a continuous image of $\omega$. 
\hfill$\square$
\smallskip 

As a consequence of Facts~4--6, we have the following result. 
\smallskip

\noindent
{\bf Fact 7.}
\begin{it}%
Let $\triple{X}{\leq^X}{\lambda^X}$ be a scattered linear ordering. 
Assume that any continuous image of $X$ is a LOTS. 
Then the following holds
\begin{itemize}
\item[{\rm(1)}]
If $x$ is in the topological closure of a nonempty subset $A$ in $X$, then  there is a countable monotone sequence of elements of $A$ converging to $x$.
\item[{\rm(2)}]
Every monotone sequence $\seqnn{ x_n }{ n \in \omega }$ converges. 
In particular, there exist both minimum and maximum in $\pair{X}{\leq}$. 
\hfill$\square$
%
\end{itemize}
\end{it}

Now we are in a position to finish the argument.

\begin{proof}[Proof of Main Theorem (for LOTS)]
Let $\equiv$ be the equivalence on $\triple{X}{\leq^X}{\lambda^X}$ defined by: 
 $$
x \equiv y 
\text{ \ if and only if  \ } %
\left\{ \begin{array}{ll}
x \leq y  
&
\text{and \  \ } [x,y] \text{ is a scattered linear ordering} 
\\
x \leq y 
&
\text{and \  \ } [y,x] \text{ is a scattered linear ordering} 
\end{array}
\right. 
$$
Note that each $\equiv$-class is closed for the topology $\lambda^X$ and each $\equiv$-class is convex and scattered for the order $\leq^X$. 
Moreover, $X/{\equiv}$, denoted by $X_1$, is a LOTS.

\smallskip 

Note that there are no consecutive classes in $X_1$, and thus $X_1$ is order-dense.
Also the map $\fnn{ \varphi }{ X }{ X/{\equiv} }$, preserving supremum and infimum, is increasing and continuous.

Let $N$ be a linear ordering and $N^c$ its Dedekind completion. 
Recall that a {\it cut} is a member of $N^c \setminus  \bigl(L \cup \{ \min(N^c) , \max(N^c) \}\bigr)$. 
For instance, in the chain $\natQ$ of rationals, the cuts are the irrationals. 

\noindent
{\bf Case 1.}
\begin{it}%
The set $\Gamma \ ( \eqdf X_1{}^c \setminus X_1 \, )$ of cuts of $X_1$ has no consecutive elements and $\Gamma$ is topologically dense in $X_1$. 
\end{it}%

So, $\Gamma$ is order--dense and $\Gamma$ has no first and no last element. Note that, in that case, $X_1$ is $0$-dimensional.
Let $c \in \Gamma$. 
Let $\seqnn{ c_\alpha }{ \alpha < \lambda }$ be a strictly increasing sequence cofinal in $(-\infty,c)$ with $\lambda$ regular. 
By Fact~4, $\lambda = \omega$. 
Since $c \in \Gamma$, 
$D \eqdf \setof{ x_\alpha }{ \alpha \in \omega }$ is a countable discrete and closed subset of $X_1$, which contradicts Fact~6. 
So, Case~1 does not occur. 

\noindent
{\bf Case 2.}
\begin{it}%
The set $\Gamma \ ( \eqdf X_1{}^c \setminus X_1 \, )$ of cuts of $X_1$ has two consecutive elements, or $\Gamma$ is not topologically dense in $X_1$. 
\end{it}%

Then there is a nonempty open interval $(u,v)$ of $X$ such that $(u,v) \cap \Gamma = \emptyset$. 
We set $X'_1 \eqdf [u,v]$. 
So $X'_1$ is connected, infinite and order--dense. 
Also $X'_1$ is a continuous image of $X_1$. 
Let $X_2$ be the quotient space of $X'_1$, obtained by identification of $u$ and $v$. 
Obviously, $X_2$ is connected. 
Since for every $t \in X_2$ the set $X_2 \setminus \{ t \}$ is connected, $X_2$ is not a LOTS. 

We have proved that $\abs{ X_1 } = 1$, that is: $X$ is a scattered linear ordering. 
By Fact~1, $X$ is a scattered topological space. 
Moreover, $X$ satisfies c.c.c., and thus $\omega_1$ and $\omega_1^*$ are not order-embeddable in the scattered linear ordering $X$. In particular, the space $X$ has only countably many isolated points. 
Also the space $X$ has no infinite and discrete subset. 
Hence the linear ordering $X$ is complete and thus the space $X$ is compact. 
We have proved that $X$ is a countable compact and scattered space, that is, $X$ is homeomorphic to $\alpha+1$ for some $\alpha<\omega_1$.

\end{proof}

\section{Basic facts on LOTS and GO-spaces}
\label{section-basic}

Let $S$ be a set, $\cmU \subseteq \powerset(S)$ and $T \subseteq S$. 
We set 
$\cmU \restriction T = \setof{ U \cap T }{ U \in \cmU }$.
Recall that a linear ordering $\pair{M}{\leq^M}$ is \dfn{complete}{complete (linear ordering)} 
whenever every subset $A$ of $M$ has the supremum $\sup^M(A)$ and the infimum $\inf^M(A)$.
In particular, there exist both the maximum $\max(M)$ and the minimum $\min(M)$ in $M$. 

\smallskip

Let $N$ be a linear ordering. 
The \dfn{Dedekind completion of $N$}{Dedekind completion of $N$},  
denoted by \dfn{$N^c$}{$N^c$: Dedekind completion of $N$}, 
is a complete linear ordering containing $N$, 
such that $N^c$ is minimal with respect to this property. 
That is, 
\begin{itemize}
\item[{\rm(D1)}]
$N^c$ is a complete chain,
\item[{\rm(D2)}]
for every $x, y \in N$: if $x <^N y$ then $x <^{N^c} y$,
\item[{\rm(D3)}]
for every $x \in N^c$ there are sets $A, B \subseteq N$ 
such that $x \not\in A \cup B$ and 
$\sup^{N^c}(A) = x = \inf^{N^c}(B)$. 
\end{itemize}
For instance, $\min(N^c) = \sup^{N^c}(\emptyset) = \inf^{N^c}(N)$. 
Notice also that the Dedekind completion $N^c$ of $N$ is 
unique up to an order-isomorphism. 

In a linear ordering $N$, a \dfn{cut in $N$}{cut (in an ordering)} is a member 
of $N^c \setminus \bigl(N \cup \{ \min(N^c) , \max(N^c) \}\bigr)$, 
\cite[2.22, 2.23]{Ro}. 
For example,  the cuts for the chain $\natQ$ of rationals are irrationals. 
Next we say that 
\dfn{$\pair{a}{b}$ are consecutive in $N$}{consecutive elements \\ in a linear ordering} 
whenever $a,b \in N$, \,$a<b$ and $(a,b)^N = \emptyset$. 
So, for any linear ordering~$N$:  
\begin{itemize}
\item[{\rm(1)}]
if $\pair{a'}{b'}$ are consecutive in $N$ 
then $\pair{a'}{b'}$ are consecutive in $N^c$, and 
\item[{\rm(2)}]
if $\pair{a'}{b'}$ are consecutive in $N^c$ then $a',b' \in N$ 
and $\pair{a'}{b'}$ are consecutive in $N$.
\end{itemize}

The following fact is well-known.
\begin{proposition}
\label{prop-klllkd.}
Let $N$ be a LOTS. 
\begin{itemize}
\item[{\rm(1)}]
$\pair{N}{\lambda^N}$ is a compact space if and only if $\pair{N}{\leq^N}$ is order complete.
\item[{\rm(2)}]
$\pair{N}{\lambda^N}$ is a connected space if and only if 
$\pair{N}{\leq^N}$ has no cuts and no consecutive elements.
\hfill$\Box$
\end{itemize}
\end{proposition}

Next we introduce some basic notions on scattered linear orders 
and scattered spaces. 
Let $N$ be a linear ordering. 
We say that $N$ is \dfn{order-dense}{order-dense ordering}, 
or \dfn{dense}{dense} if 
between two elements of $N$ there is a member of $N$.
Notice that for any dense linear order $N$, the rational chain $\natQ$ 
is  order-embeddable in $N$.
A linear order $N$ is called \dfn{order-scattered}{scattered (ordering)}
or simply scattered, 
whenever the rational chain $\natQ$ is not embeddable in $N$. 
For example, $\omega_1$ and its converse ordering $\omega_1^*$
are scattered linear ordering  (by the definition,  \dfn{$\pair{\omega_1^*}{\leq}$}{$\omega_1^*$} is the ordering 
$\pair{\omega_1}{\geq}$).

A space $Y$ is \dfn{dense-in-itself}{dense-in-itself (space)} 
if $Y$ is nonempty and has no isolated point. 
A dense-in-itself and closed subspace $Y$  of a space $Z$ is called a \dfn{perfect}{perfect space} subspace of $Z$. 
A space $X$ is called \dfn{topologically-scattered}{scattered (topology)}, 
or simply \dfn{scattered (space)}{scattered (space)}, whenever $X$ does not contain a perfect subspace,
 that is, every nonempty subset $A$ of $X$ with the induced topology has an isolated point in $A$. 
We state other well-known facts about LOTS. For completeness we include the proofs.

\begin{proposition}
\label{prop-1.3}
Let $N$ be a LOTS. 

\begin{itemize}
\item[{\rm(1)}] 
The following hold.
\begin{itemize}
\item[{\rm(a)}]
If $\pair{N}{\leq^N}$ is a scattered linear ordering
then $\pair{N}{\lambda^N}$ is topologically scattered. 
\item[{\rm(b)}]
Assume that $\pair{N}{\leq^N}$ is a complete chain, 
i.e., by Proposition~{\rm\ref{prop-klllkd.}(1)}, $\pair{N}{\lambda^N}$ is a compact space. 
Then the following are equivalent.
\begin{itemize}
\item[{\rm(i)}]
$\pair{N}{\lambda^N}$ is a scattered topological space. 
\item[{\rm(ii)}]
$\pair{N}{\leq^N}$ is a scattered linear ordering.
\end{itemize}
%
\end{itemize}
\item[{\rm(2)}] 
If $\pair{N}{\leq^N}$ is order-scattered then $\pair{N}{\lambda^N}$ is $0$-dimensional.%
\smallskip

\item[{\rm(3)}] 
If $N$ has only countably many isolated 
points then $N$ is coun\-ta\-ble.
\end{itemize}
\end{proposition}
\begin{proof}
(1) 
(a) 
Assume that $\pair{N}{\lambda^N}$ is not a scattered space. 
Let $D \subseteq N$ be a dense-in-itself subset of $N$. 
Then $\pair{D}{\leq^N {\restriction} D}$ contains an  order-dense subset, and thus $N$ is not a scattered chain. 

(b) 
Suppose that $\pair{N}{\lambda^N}$ is compact and  
 that $\pair{N}{\leq^N}$ is not order-scattered. 
Let $S \subseteq N$ be an order-dense subset of $N$, 
and let $T$ be its topological closure in $\pair{N}{\lambda^N}$. 
Then $T$ has no isolated points, i.e. $T$ is dense-in-itself. 
By compactness, $T$ is compact and thus $T$ is perfect.
Hence $\pair{N}{\lambda^N}$ is not a scattered space.

We prove a little bit more. 
We have $N = N^c$.
Let $M$ be a linear ordering and let $M^c$ be its the Dedekind completion.
The following are equivalent:~ 
(i)~$M$ is a scattered chain, 
(ii)~$\natQ$ is not order-embeddable in $M$,
(iii)~$\natQ$ is not order-embeddable in $M^c$, and 
(iv)~$M^c$ is a scattered chain. 
Now since $\pair{M^c}{\leq^{M^c}}$ is a complete chain, $\pair{M^c}{\lambda^{M^c}}$ is a compact space. 
Therefore the previous items are equivalent to each of the following 
(v)~$\natQ$ is not order-embeddable in $M^c$, and 
(vi)~$\pair{M^c}{\lambda^{M^c}}$ is a scattered space.

\smallskip

(2) 
In the proof of Part~(1), we have seen that if $N$ is a scattered chain, 
then its Dedekind completion $N^c$ is a scattered chain and thus  
$N^c$, considered as a LOTS, is compact and topologically-scattered.
Therefore $N^c$ is $0$-dimensional, and thus $N$ is also $0$-dimensional.
\smallskip

(3) 
By the hypothesis, the set $S \eqdf \rfs{Iso}(N)$ of isolated points in $N$ 
is countable. 
Since $N$ is a scattered space, $S$ is topologically dense in $N$. 
Since $S$ is a chain, 
by the proof of Part~(1), the chain $S^c$ is scattered. 
We claim that $S^c$ is countable. 
This is so, because if $S^c$ is an uncountable scattered chain,
then $\omega_1$ or $\omega_1^*$ is order-embeddable in $S^c$ 
and thus the same holds for $S$, that is, $S$ is uncountable contradicting our assumptions. 

Next, since $S^c$ is countable, $N$ must be countable.  
Indeed, there exists a continuous (increasing) map $f$ from $N$ into $S^c$ 
such that $ | f^{-1}(x) |\leq 2$ for any $x \in N$. 
(That is, the case if 
$N \eqdf \omega+1+1+\omega^*$, 
and thus $S \eqdf \rfs{Iso}(N) = \omega+\omega^*$ 
and $S^c = \omega+1+\omega^*$.)
\end{proof}

\noindent
\begin{remark}\label{R2}
{\rm 
(1)
Proposition~\ref{prop-1.3}(1)(a) is not reversible. 
As an a example, consider the  lexicographic sum
$N \eqdf \sum_{q \in \natQ} \natZ_q$ 
of copies $\natZ_q$ of (the the chain of integers) $\natZ$, 
indexed by the chain of rationals $\natQ$. 
Then $N$ is a non-scattered linear ordering, but $N$ is a topological discrete LOTS and thus $N$ is a scattered topological space. 

(2) 
Recall that the Dedekind completion of $M \eqdf (0,1) \cap \natQ$ is $M^c = [0,1]^\natR$. 
On the other hand, $M$ is topologically dense in the Cantor set $2^\omega$  (considered as a subset of $\natR$). 

(3)
Let $X \eqdf \setof{ 1/n }{ n>0 } \cup \{-1\} \subset \natR \eqdf L$.
Then $\pair{X}{\lambda^X}$ is compact, 
but $\pair{X}{\tau^X}$ is infinite and discrete. 
\hfill$\square$
}
\end{remark}
\smallskip

\begin{itemize}
\item 
Let $\quaple{X}{\tau^X}{L}{\leq^L}$ be a GO-space.
Then $\lambda^L {\restriction} X \subseteq \tau^X$.
\end{itemize}
Indeed, let $a < b$ in $X$. 
So $(a,b)^X \in \lambda^X$ and 
thus, by the definition, $(a,b)^X = (a,b)^N \cap X \in \tau^X$. 

The following result is well-known. 
For completeness we include its proof.

\begin{proposition}
\label{prop-fqljdlfk}  
Let $\pair{X}{\tau^X}$ be a GO-space. 
Let $\pair{L}{\leq^L}$ be such that
\newline 
$\quaple{X}{\tau^X}{L}{\leq^L}$ is a GO-space.
Without loss of generality we may assume that $L$ satisfies: 
\begin{itemize}
\item[{\rm(H1)}]
$X$ is topologically dense in $\pair{L}{\lambda^L}$;
\item[{\rm(H2)}]
$\pair{L}{\leq^L}$ is a complete linear ordering.
\end{itemize}
\end{proposition}
\begin{proof}
The proof follows from the following two facts.
\smallskip

\noindent
\begin{it}%
Fact 1. Let $\pair{N}{\leq^N}$ be a linear ordering and $\pair{N^c}{\leq^{N^c}}$ be its Dedekind completion. 
Then $\lambda^N = \tau^{N^c} {\restriction} N$. 
That is, the order topology $\lambda^N$ is the induced topology of $\tau^{N^c}$ on $N$.
\end{it}

\proof
Let $c \in N^c$. 
Then $c$ is a cut if and only if 
$c \not\in N \cup \{ \min(N), \max(N) \}$ and 
$c$ has no predecessor nor a successor in $N^c$. 
Obviously, $\lambda^N \subseteq \tau^{N^c} {\restriction} N$.
Next let $(a,b)^{N^c} \in \lambda^{N^c}$ be an open convex set in $\pair{N^c}{\leq^{N^c}}$. 
So $a,b \in N^c$. 
If $a \not\in N$ then $a = \inf ( \setof{ a' \in N }{ a' > a } )$ and 
if $b \not\in N$ then $b = \inf ( \setof{ b' \in N }{ b' < b } )$. 
So $(a,b)^{N^c} \cap N$ is an union of open convex sets $(a',b')^N$ where $a',b' \in N$.
\hfill$\square$

\smallskip

\noindent
\begin{it}%
Fact 2. 
Let $\quaple{X}{\tau^X}{N}{\leq^N}$ be a GO-structure such that 
$\pair{N}{\tau^N}$ is a complete ordering. 
Let $L$ be the topological closure of $X$ in the space $\pair{N}{\lambda^N}$. 
Then $\tau^X = \lambda^L {\restriction} X$. 
That is, $\quaple{X}{\tau^X}{L}{\leq^N \! \!  {\restriction} L}$ is a GO-structure. 
\end{it}
\proof 
It suffices to show that for every $a<b$ in $N$ there are $a'<b'$ in $L$ such that $(*$):~$(a',b')^L \cap X = (a,b)^N \cap X$. 
Fix $a<b$ in $N$.
If $a \in X$ ($b \in X$) set $a'=a$ ($b'=b$). 
Next suppose that $a \in N \setminus L$. 
Since $N$ is a complete ordering and $N \setminus L$ is open in $L$, there is a (maximal) open convex set $(\alpha,\alpha')^N$ in $N$ such that $a \in (\alpha,\alpha')^N$, $(\alpha,\alpha')^N \cap L = \emptyset$ and $\alpha, \alpha' \in L$; and we set $a' = \alpha$.

Similarly, suppose that $b \in N \setminus L$. 
Again, since $N$ is complete and $N \setminus L$ is open in $L$, there is a (maximal) open convex set $(\beta,\beta')^N$ such that $b \in (\alpha,\beta')^N$, $(\beta,\beta')^N \cap L = \emptyset$ and $\beta, \beta' \in L$; and we set $b' = \beta'$. 
Now obviously $(a',b')^L $ is as required in~($*$).
\hfill$\square$ 

\smallskip

Now let $\quaple{X}{\tau^X}{N}{\leq^N}$ be a GO-structure.
By Fact~1 we may assume that $N$ is a complete ordering. 
Finally, the result follows from Fact~2. 
\end{proof}
\begin{remark}\label{R4}
\begin{rm}
(1)
In general $\lambda^X \subsetneqq \tau^X$.
For example, consider $L = \omega_1$ and let $\rfs{Lim}$ be the set of all
countable limit ordinals.
We set $X = \omega_1 \setminus \rfs{Lim}$. 
We have: 
\\
(i) $\tau^X$ is the discrete topology, 
\\
(ii) $Y$ is topologically dense in 
$\pair{\omega_1}{ \lambda^{\omega_1} }$, and 
\\
(iii) $X$ is order-isomorphic to $\omega_1$ and thus $\pair{Y}{\lambda^{\omega_1}}$ is homeomorphic 
to the ordinal space~$\omega_1$. 
(2)
The one-point compactification of an uncountable discrete space is not a GO-space (by Lemma~\ref{lemma-2.1}), and 
there is a countable space which is not a GO-space (by Lemma~\ref{lemma-2.4}). 
\hfill$\square$
\end{rm}
\end{remark}

For completeness we recall the proof of the following fact.
\begin{proposition}{\rm\cite[Lemma 6.1]{L}}
\label{prop-dfKSXCB.}
Let $\quaple{X}{\tau^X}{L}{\leq^L}$ be a GO-structure satisfying {\rm (H1)} and {\rm (H2)}. 
So $\pair{ X }{ \, \leq^L \! {\restriction} X }$ is a linear ordering.

{\rm(1)}
If $\pair{X}{\tau^X}$ is a compact space then $X = L$ and $\tau^X = \lambda^X$.

{\rm(2)}
If $\pair{X}{\tau^X}$ is a connected space then $\tau^X = \lambda^X$.
\end{proposition}
\begin{proof}
(1) 
Since $X$ is compact then $X$ is closed in $\pair{L}{\lambda^L}$. 
By~(H1), $X = L$ and thus $\tau^X = \lambda^X$. 

(2) 
Next suppose that $\pair{X}{\tau^X}$ is connected. 
Then $X$, considered as the LOTS $\pair{X}{\lambda^X}$, has no consecutive point and no cut because for each final subset of $\pair{ X }{ \leq^L {\restriction} X }$ is closed in $\pair{X}{\tau^X}$. 
Therefore, $ X \setminus \{ \min(L), \max(L) \} = L \setminus \{ \min(L), \max(L) \}$ and thus  $\tau^X = \lambda^X$.
\end{proof}


\section{Proof of Main Theorem}
\label{proof main theorem}
As one of the main parts of Main Theorem~\ref{main-thm} Proposition~\ref{prop-kfjdsluzreo} implies that it suffices to assume that $\pair{L}{\leq}$ is a scattered linear order. 
To prove this result, we use Corollary~\ref{cor-hfdsk}
(in \S\ref{subsection counterexample}) which says that a GO-space does not contain an infinite countable relatively discrete closed subset. 

Recall that $\pair{X}{\tau}$ is a GO-space means that 
$\quaple{X}{\tau}{L}{\leq}$ is a GO-structure.
So, we have $X \subseteq L$. 
Also for simplicity $\pair{X}{\tau}$ is denoted by $X$.
In the sequel, by Proposition~\ref{prop-fqljdlfk}, 
we assume that the GO-structure $\quaple{X}{\tau}{L}{\leq}$
satisfies properties (H1) and (H2).
\subsection{A hereditarily GO-space satisfies c.c.c. property}
\label{subsection preliminary}

\begin{lemma}\label{lemma-2.1}
	
Assume that $\quaple{X}{\tau}{L}{\leq}$ is a $0$-dimensional hereditarily GO-structure. 
Then

\begin{itemize}
\item[{\rm(1)}]
$X$ satisfies c.c.c. property\,.

\item[{\rm(2)}]
The linear orderings $\omega_1$ and $\omega_1^*$ 
are not order-embeddable in $X$. 
\end{itemize}
\end{lemma}
\begin{proof}
(1) 
Since $\pair{X}{\tau}$ is  0-dimensional, any nonempty open subset of $\pair{X}{\tau}$ contains a clopen convex subset of the form $(a,b)^X \eqdf (a,b)^L \cap X$ where $a, b \in L$. 
By contradiction, assume that  $\setof{ U_i }{i \in I }$ is an uncountable family 
of pairwise nonempty clopen convex subsets of $X$. 
So each $U_i$ is of the form $(a_i,b_i)^X$ with $a_i, b_i \in L$. 

Fix $i_0 \in I$.
Let $U \eqdf \bigcup\SETOF{U_i}{i \in I \setminus \{i_0\} }$ and  \dfn{$\sim$}{$\sim$} be the equivalence relation on $X$ defined as follows: 
$x \sim y$ if and only if $x,y \in Y {\setminus} U$ 
or there is $i \in I$ such that $x, y \in U_i$. 
Denote by $X'$ the set $X/{\sim}$ 
and by  $\fnn{ f }{ X }{ X' }$ the quotient map. 
We endow $X'$ with the quotient topology $\tau'$ on $X'$. 
So $V' \in \tau'$ if and only if $f^{-1}[V'] \in \tau$. 
In particular, $u_i \eqdf f[U_i]$ is an isolated point in $X'$ 
for any $i \in I \setminus \{i_0\}$. 
Setting $u = f[X \setminus U]$ we have 
$X' = \{u\} \cup \bigcup\SETOF{u_i}{i \in I \setminus \{i_0\} }$ 
and the set  
$$
\SETOF{ \{u_i\} }{ i \in I \setminus \{i_0\} } 
\ \cup \ \SETOF{ X' \setminus \{u_i\} }{ i \in I \setminus \{i_0\} } 
$$ 
is a subbase of a topology $\tau''$ on $X'$ satisfying: 
\begin{itemize} 
\item[{\rm(1)}]
$\tau'' \subseteq \tau'$ and thus 
$\fnn{ f }{ \pair{X}{\tau} }{ \pair{X'}{\tau''} }$ is continuous,  
\item[{\rm(2)}] 
$\pair{X'}{\tau''}$ is compact (this follow from the definition of $\tau''$), and 
\item[{\rm(3)}]
$\pair{X'}{\tau''}$ is homeomorphic to the one-point compactification of the uncountable discrete set. 
This is so because $u$ is the unique accumulation point of $\pair{X'}{\tau''}$.
\end{itemize}
We show that $\pair{X'}{\tau''}$ is not a GO-space.
By contradiction, suppose that
\newline
 $\quaple{X'}{\tau''}{L'}{\leq'}$ is a GO-structure. 
Since $\pair{X'}{\tau''}$ is compact, by Proposition~\ref{prop-dfKSXCB.}(1), 
$L' = X'$ and $\tau'' = \lambda' \eqdf \lambda^{\leq'}$.
So it suffices to prove that 
\begin{itemize}
\item[{\rm(4)}] 
$\pair{X'}{\tau''} \eqdf \pair{X'}{\lambda'}$ is not a LOTS. 
\end{itemize}
Assume that $\pair{L'}{\leq'}$ is a chain.   
Hence, for instance, $(-\infty,u)^{L'}$ is uncountable.
Consider any $y \in (-\infty,u)$ such that $(-\infty,y)^{L'}$ is infinite. 
By the definition, $(-\infty,y]^{L'}$ is infinite, discrete, closed and thus compact, that contradicts the fact that $u \not\in (-\infty,y]^{L'}$. 
We have proved that $X$ satisfies c.c.c. property\,.
\smallskip

(2) follows from Part~(1).
\end{proof}
Now remind the classical result which is due to
Mazurkiewicz and Sierpi\'{n}ski (for example, see ~\cite[Theorem~17.11]{K}, \cite[Ch. 2, Theorem~8.6.10]{S}).  
\begin{lemma}
\label{lemma-2.2}
Every topologically scattered compact and countable space
is homeomorphic to a countable and successor ordinal space.
\end{lemma}

\subsection{A hereditarily GO-space has no countable closed and relatively discrete subsets}
\label{subsection counterexample}

\begin{lemma}
\label{lemma-2.3}
Let $M$ be an order-scattered LOTS. 
If $M$ contains a closed and countable relatively discrete subset $D$,
then $D$ is a continuous image of~$M$.
\end{lemma}

\begin{proof}
First we introduce a new definition.
Let $Y$ be a topological space. 
For a family $\cmV$ of pairwise disjoint subsets of $Y$, we denote by \dfn{${\rm acc}(\cmV)$}{${\rm acc}(\cmV)$} the set of \dfn{accumulation points of $\cmV$}{accumulation points of a family of sets}. By definition,
$x \in {\rm acc}(\cmV)$ if and only if for every neighborhood $W$ of $x$ the set $\setof{ V \in \cmV }{ V \cap W \neq \emptyset }$
is infinite. 
So if $Z \subseteq Y$, ${\rm acc}(Z) = {\rm acc}\bigl(\SETOF{ \{z \} }{ z \in Z } \bigr)$.
We  say that a subset $D$ of a space $Y$ is \dfn{strongly discrete}{strongly discrete (subset)} whenever $\cmU_D$ satisfies 
${\rm acc}(D) = {\rm acc}(\cmU_D)$.
The next result is well-known. 
For completeness we recall its proof.

\smallskip

\noindent
\begin{it}%
Fact~1.
Let $M$ be a $0$-dimensional LOTS and $D \subseteq M$.
If $D$ is relatively discrete then $D$ is strongly discrete.
\end{it}
\proof
For $d \in D$ let $U_d \eqdf (a_d, b_d)^L$ be a clopen convex set such that $D \cap U_d = \{d\}$. 
Note the following property ($*$):~if $d < d'$ then $x < x'$ for every $x \in U_d$ and $x' \in U_{d'}$. 
We set $\cmU_D = \setof{ U_d }{ d \in D }$.
Obviously, ${\rm acc}(D) \subseteq {\rm acc}(\cmU_D)$.
Conversely, let $x \in {\rm acc}(\cmU_D)$ and $V$ be a neighborhood of $x$. 
We may assume that $V$ is of the form $(a,b)$ with $a<b$ in $M$. 
Therefore, by~($*$), there are infinitely many $d_i$'s such that $d_i \in U_{d_i} \subseteq V$, and thus $x \in {\rm acc}(\cmU_D) \subseteq {\rm acc}(D)$.
\hfill$\square$

\smallskip

Since $\pair{M}{\leq^M}$ is an order-scattered LOTS,  
by Proposition \ref{prop-1.3}(2), $\pair{M}{\lambda^M}$ is $0$-dimensional. 
By Fact~1, 
let $\cmU_D \eqdf \setof{ U_d }{ d \in D }$ be a family of clopen convex subsets of $M$ such that $U_d \cap D = \{ d \}$ for $d \in D$ and ${\rm acc}(D) = {\rm acc}(\cmU_D)$. 
Since $D$ is closed, ${\rm acc}(D) = \emptyset$ and thus
$\bigcup \cmU_D$ is a clopen subset of $M$.

Let $d_0 \in D$ be fixed and 
$\cmU_{D {\setminus} \{d_0\}} = \setof{ U_d }{ d \in D {\setminus} \{d_0\} }$. 
So $\bigcup \cmU_{D {\setminus} \{d_0\}} = \bigcup \cmU_D \setminus U_{d_0}$ is a clopen subset of $M$.
Let \dfn{$\approx$}{$\approx$} be the equivalence relation on $M$ defined by $x \approx y$ whenever $x, y \in U_D$ for some $d \in D \setminus \{d_0\}$, or 
$x,y \in M \setminus \bigcup \cmU_{D {\setminus} \{d_0\}}$. 
For each $x \in M$ there is an unique $f(x) \in D$ such that $f(x) \approx x$. 
It is easy to check that the mapping $f\colon M \to D$ is onto and that $f$ is continuous: 
this is so, because $f^{-1}(d)$ is clopen in $M$ for any $d \in D$.
\end{proof}
\smallskip

Our next result, which is apparently well-known, strengthens Fact 6 from the proof of Corollary~\ref{cor-456789123nn}.
Consider again $\omega$ as a discrete space. 
Let $\cmU$ on $\omega$ be a non-principal ultrafilter on $\omega$ and let $*$ be a new element with $* \not\in \omega$.
The space $\natN^* := \omega \cup \{*\}$ is equipped with the topology induced from the Stone--\v Cech compactification $\beta \natN$.

\begin{lemma}
\label{lemma-2.4} 
The countable space $\natN^*$ is not a GO-space.
\end{lemma}

\begin{proof}
Recall that the character $\chi(X)$ of the infinite topological space $X$ is the supremum of cardinalities 
of minimal local neighborhood bases of all points in $X$.
\noindent

\begin{it}%
Fact 1.
Let $\quaple{X}{\tau^X}{L}{\lambda^L}$ be a GO-structure satisfying {\rm(H1)} and {\rm(H2)}. 
If $X$ is countable then $\chi(\pair{X}{\tau^X}) = \aleph_0$.
\end{it}
\proof 
Since $X$ is a countable subset of $L$ and $X$ is topologically dense in $L$, 
the chain $L$ is order-embeddable in the segment  $[0,1]$ of $\natR$. 
Since $\natQ \cap [0,1] \subseteq [0,1]$ we have $\chi([0,1]) = \aleph_0$. 
So, $\aleph_0 \leq \chi(X) \leq \chi(L) \leq \chi([0,1]) = \aleph_0$.  
\hfill$\square$%
\smallskip

\noindent
Now it suffices to remind a well-known fact that 
$\chi(\natN^*) > \aleph_0$ (see \cite[3.6.17]{E}).

Therefore, $\natN^*$ is not a GO-space.
\end{proof}

As a consequence of the above results~\ref{lemma-2.3} 
and~\ref{lemma-2.4}, we have: 

\begin{corollary}
\label{cor-hfdsk}
Let $X$ be a hereditarily GO-space. 
Then $X$ does not contain an infinite countable relatively discrete closed subset. 
\hfill$\Box$
\end{corollary}


\subsection{A hereditarily GO-space comes from a scattered linear ordering}
\label{subsection scatteredness}

Let GO-structure $\quaple{X}{\tau^X}{L}{\leq^L}$ satisfies the conditions: 
\begin{it}%
\begin{itemize}
\item[{\rm(H1)}]
$X$ is topologically dense in $\pair{L}{\lambda^L}$.
\item[{\rm(H2)}]
$\pair{L}{\leq^L}$  is complete, meaning that 
$\pair{L}{\lambda^L}$ is a compact LOTS.
\end{itemize} 
\end{it}
Hence, by Proposition~{\rm\ref{prop-1.3}(1)(b)},
\begin{itemize}
\item[{\rm(H3)}]
For any $u<v$ in $L$: 
$[u , v]^L$ is order-scattered if and only if $[u , v]^L$ is  topologically-scattered.
\end{itemize}
Let $\pair{X}{\tau}$ be a hereditarily GO-space,  
meaning that $\quaple{X}{\tau^X}{L}{\leq^L}$ is a hereditarily GO-structure 
satisfying (H1)--(H3). 
For simplicity denote the space $\pair{X}{\tau^X}$ by $X$. 
We shall show that $\pair{L}{\leq^L}$ is order-scattered (Proposition~\ref{prop-kfjdsluzreo}), and thus, by Proposition~{\rm\ref{prop-1.3}(1)(b)}, $\pair{L}{\lambda^L}$ is topologically--scattered. 
Therefore $X$, as subset of $L$, is also topologically--scattered. 

Let \dfn{$\equiv^L$}{$\equiv^L$} be the equivalence relation on $L$ defined as follows.
For $x, y \in L$, we set $x \equiv^L y$ if 
$x \leq y$ and $[x , y]^L$ is an order-scattered subset of $L$, or
$y \leq x$ and $[y , x]^L$ is an order-scattered subset of $L$. 
Note that, by~(H3), 
in the definition of $\equiv^L$ we have: $[u , v]^L$ is order-scattered if and only if $[u , v]^L$ is topologically-scattered.

Now, each equivalence class is an order-scattered  
and convex subset of the LOTS~$\pair{L}{\leq}$ and each equivalence class is closed in $\pair{L}{\lambda}$. 
(\,$\equiv^L$ is standard (see the proof of Theorem~19.26, in~\cite{KKLP}\,)). 

We denote $L/{\equiv^L}$ by $L_1$ and  
by $\fnn{ \pi }{ L }{ L_1 }$ the projection map. 
So $\pi$ is increasing and thus $\pi$ induces a linear order $\leq_1$ on $L_1$. 
\begin{lemma}
\label{lemma-idsqp}
The linear ordering $\pair{L_1}{\leq_1}$ has the following properties.
\begin{itemize}
\item[{\rm(1)}]
$L_1$ is a complete linear ordering and the quotient topology on $L_1$ is the order topology $\lambda_1 \eqdf \lambda^{\leq{L_1}}$. 
\item[{\rm(2)}]
$L_1$ is order-dense, i.e. $L_1$ has no consecutive elements.
\item[{\rm(3)}]
$L_1$ is a compact and dense-in-itself space.
\item[{\rm(4)}]
$L_1$ is a connected space. 
\end{itemize}
%
\end{lemma}
\begin{proof}
(1) This part follows from the fact that $L$ is complete and $\pi$ is increasing and onto.

(2)--(3) 
Notice first that  $\pair{L_1}{\leq_1}$ is a complete chain, and thus 
$\pair{L_1}{\lambda_1}$ is compact.
Secondly, there are no consecutive $\equiv^L$-classes in $L$,  
this is so because the union of two consecutive $\equiv^L$-classes 
is an $\equiv^L$-class. 
Hence $L_1$ has no consecutive elements.

Therefore, $\pair{L_1}{\leq_1}$ is a dense chain  
and thus $L_1$ is dense-in-itself. 

(4) 
By Part~(2), $L_1$ has no consecutive elements.  
Also since $L_1$ is a complete chain, $L_1$ has no cuts.
So, by Proposition~\ref{prop-dfKSXCB.}, $L_1$ is a connected.
\end{proof}

Now we recall that $X \subseteq L$ and that for $x<y$ in $L$: 
\begin{itemize}
\item[{\rm(1)}]
$[x , y]^L$ is a scattered subspace of $\pair{L}{\lambda}$ if and only if $[x , y]^L$ is a scattered subchain of $\pair{L}{\leq}$, and  
\item[{\rm(2)}]
if $[x , y]^L$ is a scattered subspace of $L$ 
then $[x , y]^X \eqdf [x , y]^L \cap X$ is a scattered subspace of $X$ 
(but not vice-versa). 
\end{itemize}
Now the relation $\equiv^L$ induces an equivalence relation $\equiv^X$ on $X$, setting for \hbox{$x,y \in X$:} 
$$
x \equiv^X y \text{ \ if and only if \ } x \equiv^L y \, .
\dfn{}{$\equiv^X$}
$$
We denote by $X_1$ the space $X/{\equiv^X}$. 

\begin{lemma}
\label{lemma-hdfskqyzreio} 
The following hold for $\quaple{X_1}{\tau^{X_1}}{L_1}{\leq^{L_1}}$.
\begin{itemize}
\item[{\rm(1)}]
The continuous inclusion embedding $X \subseteq L$ induces 
a continuous inclusion embedding $X_1 \subseteq L_1$. 
\item[{\rm(2)}]
$X_1$ is a topologically dense subset of $L_1$ for the topology $\lambda_1$ 
(on $L_1$).
\item[{\rm(3)}] 
$X_1$ considered as subordering of $L_1$ has no consecutive points. 
\item[{\rm(4)}] 
$\quaple{X_1}{\tau_1}{L_1}{\leq_1}$ is a GO-structure.
\end{itemize}
\end{lemma}
\begin{proof}
(1)
Let $\fnn{ \pi }{ L }{ L_1 }$ be the projection map. 
Then $\pi[X] \eqdf X_1 \subseteq L_1$ and the embedding $X_1 \subseteq L_1$ is continuous. 

(2)
Since $X$ is a topologically dense subset of $L$,
$X_1$ is topologically dense in $L_1$. 

(3) 
Since $X_1$ is topologically dense in $L_1$, 
if $a_1<b_1$ are consecutive elements in $X_1$ 
then $a_1<b_1$ are also consecutive in $L_1$.
This contradicts Lemma~\ref{lemma-idsqp}(2).

(4) follows from the definitions.
\end{proof}

We have seen that $\quaple{X_1}{\tau_1}{L_1}{\leq_1}$ is a GO-structure
with the properties (H1), (H2), and the properties
(1)--(4) of Lemma~\ref{lemma-idsqp} and (1)--(3) of Lemma~\ref{lemma-hdfskqyzreio}.

\begin{proposition}
\label{prop-kfjdsluzreo} 
Let $\quaple{X}{\tau}{L}{\leq}$ be a hereditarily GO-space satisfying {\rm(H1)} and {\rm(H2)}. 
Then $\pair{L}{\leq}$ is a scattered chain. 
\end{proposition}
\begin{proof} 
Now, with the above notations of \S\ref{subsection scatteredness}, we consider the GO-space
\newline 
$\quaple{X_1}{\tau^{X_1}}{L_1}{\leq^{L_1}}$ instead of the GO-space $\quaple{X}{\tau^{X}}{L}{\leq^{L}}$.
Let  
$$
\Gamma_1 = L_1 \setminus 
\bigl( X_1 \cup \{ \min(L_1),  \max(L_1) \} \bigr) \, . 
\dfn{}{$\Gamma_1$: set of cuts of $X_1$ in $L_1$}
$$
That is, ``$\Gamma_1$ is the set of cuts of the chain $\pair{ X_1 }{\leq^{X_1}}$ considered as linear sub-ordering order of $\pair{ L_1 }{\leq^{L_1}}$''. 
An obvious characterization of the elements of $\Gamma_1$ is stated in the following fact. 
\smallskip

\noindent
\begin{it}%
Fact 1.
We have that $\gamma \in \Gamma_1$ if and only if $\gamma$ defines two nonempty clopen sets, namely $(-\infty,\gamma)^{X_1}$ and $(\gamma, +\infty)^{X_1}$, for the induced topology $\tau^{X_1}$. 

Therefore 
$\pair{X_1}{\tau^{X_1}}$ is a connected space if and only if $\Gamma_1$ is empty and $X_1$ has no consecutive elements. 
\hfill$\square$%
\end{it}
\smallskip

To prove Proposition~\ref{prop-kfjdsluzreo}, we distinguish two cases, 
and in fact we prove that $\abs{X_1} = 1$. 
This implies that $\pair{L}{\leq^L}$ is order-scattered. 
\smallskip

\noindent
\begin{it}%
Case 1. 
$\Gamma_1$ has no consecutive elements 
and $\Gamma_1$ is topologically dense in $L_1$ 
for the order topology $\lambda_1$. 
\end{it}
\smallskip
\newline
So the set $\Gamma_1$ is a dense linear order and 
every nonempty open convex subset of $L_1$ contains a cut. 
Since $X_1$ is topologically dense in $L_1$, any nonempty open convex subset of $L_1$ contains a cut and thus, in that case, 
\begin{itemize}
\item 
$\pair{X_1}{\tau_1}$ is 0-dimensional. 
\end{itemize}
Let $c \in \Gamma_1$. 
Since $\Gamma_1$ has no first element,  
let $\seqn{ c_\alpha }{ \alpha < \lambda }$ be a cofinal strictly increasing 
sequence in $(-\infty,c)^{X_1}$ 
where $\lambda$ is an infinite regular cardinal, 
that is, for every $x \in (-\infty,c)^{X_1}$ there is $\alpha$ 
such that $x \leq c_\alpha$. 
\smallskip

\noindent
\begin{it}%
Fact 2.
$\lambda = \omega$.
\end{it}
\proof 
If not then $\omega_1$ is order-embeddable in $\Gamma_1$.
Choose $x_\alpha \in (c_\alpha,c_{\alpha+1}) \cap X_1$ for any $\alpha$, 
then $\seqn{ x_\alpha }{ \alpha < \lambda }$ is 
a sequence in $X_1$, order-isomorphic to $\omega_1$.
Since $\pair{X_1}{\tau_1}$ is 0-dimensional 
(but not necessarily an interval space), 
this contradicts Lemma~\ref{lemma-2.1}(2). 
We have proved that $\lambda = \omega$. 
\hfill$\square$
\smallskip

Keeping the same notations, we have 
($*$):~the set $D \eqdf \setof{ x_\alpha }{ \alpha \in \omega }$ is countable and relatively discrete.
Also since $c$ is a cut, 
($**$):~$D$ is closed in $X_1$. ($*$) together with ($**$) contradicts Corollary \ref{cor-hfdsk}.
Therefore, Case~1 does not occur. 
\smallskip
\newline
\begin{it}
Case 2. 
Not Case 1.
\end{it}
\smallskip
\newline
This implies that 
either $\Gamma_1 \eqdf L_1 \setminus 
\bigl( X_1 \cup \{ \min(L_1),  \max(L_1) \} \bigr)$ 
has two consecutive elements in $L_1$ 
or $\Gamma_1$ is not topologically dense in $L_1$.
Anycase, 
there is an infinite open interval $(u,v)^{L_1}$ of $L_1$ 
(with $u<v$ in $L_1$)
that does not contain a member of~$\Gamma_1$. 
Recall that we have the following properties.
\begin{itemize}
\item[{\rm(P1)}]
The elements of $L_1$ are exactly the $\equiv^L$-classes of $L$,  
and  
\item[{\rm(P2)}]
$\pair{L_1}{\leq_1}$ is a dense linear order. 
\end{itemize}
Since $(u,v)^{L_1}$ is infinite, we may assume that $u, v \not\in \Gamma_1$.
So, we have the additional properties.
\begin{itemize}
\item[{\rm(P3)}]
$[u,v]^{L_1} \cap \Gamma_1 = \emptyset$, and thus 
$[u,v]^{L_1} = [u,v]^{X_1}$.
\item[{\rm(P4)}]
$[u,v]^{L_1}$ is infinite.
\item[{\rm(P5)}]
$X_1$, and thus $[u,v]^{X_1}$, has no consecutive elements (Lemma~\ref{lemma-hdfskqyzreio}(3)). 
\item[{\rm(P6)}]
Hence, by (P3),  (P5) and Proposition~\ref{prop-klllkd.}(2), 
$[u,v]^{X_1}$ is a connected space.
%
\end{itemize}
Next we prove that $\abs{X_1}= 1$. 
By contradiction, assume that $\abs{X_1}> 1$. 
We consider the equivalence relation \dfn{$\approxeq$}{$\approxeq$} on $X_1$
which identifies all elements of $(-\infty,u] \cup [v,+\infty)X$.
So $X_1/{\approxeq}$, denoted by $X_2$, is a continuous image of~$X_1$. 
Also  $X_2$ is a connected space and, by~(P4), $X_2$  is an infinite continuous image of $X$.
\smallskip

\noindent
\begin{it}%
Fact 3. 
Let $\quaple{Y}{\tau^Y}{M}{\leq^M}$ be a GO-structure such that 
$\pair{Y}{\tau^Y}$ is connected. 

Then, for every $y \in Y$: if $y$ is not the minimum nor the maximum of $Y$  
(if they exist) 
then the subspace $Y \setminus \{ y \}$ is not a connected space.
\end{it}
\proof 
By Proposition~\ref{prop-dfKSXCB.}(2), $\tau^Y = \lambda^Y$.
Let $y \in Y$ be such that $y$ is not the minimum nor the maximum of $Y$. 
Set $U = (\infty,y)^M$ and $V = (y,\infty)^M$.
By the definition, $U$ and $V$ are open subsets of $M$. 
Hence $U \cap Y$ and $V \cap Y$ define a partition of $Y$ into two nonempty open sets of $Y$. 
We have proved Fact~3. 
\hfill$\square$%
\smallskip

We claim that 
\smallskip

\noindent
\begin{it}%
Fact 4.
For every $x \in X_2$, the subspace $X_2 \setminus \{ x \}$ is connected. 
\end{it}
\proof
First recall that $X_2$ is a connected space. 
Also we can define $X_2$ as follows: 
$X_2$ is the quotient of $[u,v]^{X_1}$ identifying $u$ and $v$. 
Fact~4 follows from the claim that
$[u,v]^{X_1}$ is a connected interval subspace of $X_1$. 
%
\hfill$\square$
\smallskip

Now, from~(P6) it follows that: $X_2$ is connected,  
and by Fact~4: for any $x \in X_2$ the space $X_2 \setminus \{ x \}$ is connected. 
Hence, by Fact~3, the space $X_2$ is not a GO-space. 
So $X_1$ and thus $X$ is not a GO-space.
 
In other words, by (P1) and (P2), if $X$ (or equivalently $L$) has more than one $\equiv^{L}$-class, then $X_2$ is a continuous image $X$ and $X_2$ is not a GO-space. This contradicts the fact that $X$ is a hereditarily GO-space. 
We have proved that $\abs{X_1}= 1$. 

Further, $\abs{X_1}= 1$ 
means that $X_1$ consists of exactly one $\equiv^X$-class, 
or, equivalently, there is the unique $\equiv^L$-class in $L$. 
Since for $x<y$ in $X$: $x \equiv^L y$ if and only if  $[x,y]^L$ is a scattered chain, the chain $L$ is scattered.
\end{proof}


\subsection{End of the proof of Main Theorem}
\label{subsection proof of thm} 

Let $\quaple{X}{\tau}{L}{\leq}$ be a hereditarily GO-space.
We prove first that $X$ is countable. 
By Proposition~\ref{prop-kfjdsluzreo}, $\pair{L}{\leq}$ is a scattered chain.
Since $\pair{L}{\lambda}$ is compact and topologically-scattered,
by Lemma~\ref{lemma-2.1}, $\pair{X}{\tau}$ satisfies c.c.c. property, 
so $X$ has only countably many isolated points. 
Denote by $\rfs{Iso}(Y)$ the set of isolated points of $Y$.
Since $X$ is topologically dense in $L$, we have $\rfs{Iso}(L) = \rfs{Iso}(X)$ 
and thus $\rfs{Iso}(L) = \aleph_0$.
Therefore, by Proposition~\ref{prop-1.3}(3), $L$, and thus $X$ is also countable.

Next, by Lemma~\ref{lemma-2.3}(3), 
the space $X$ does not contain a countable relatively discrete set. 
Since $X$ is countable, $X$ is closed under supremum and infimum in $L$, 
and thus $X$, as a linear order, is complete. 
We have seen that $L$ and $X$ are compact and countable. 
Finally, since $X$ is topologically-scattered, by Lemma~\ref{lemma-2.2}, $X$ is homeomorphic to the LOTS $\alpha+1$ where $\alpha$ is a countable ordinal.

We have proved Main Theorem~\ref{main-thm}.

\end{document}